\documentclass[a4paper, 12pt, reqno]{amsart}

\usepackage{fancyhdr}
\usepackage{amssymb,amsmath}
\usepackage[dvips]{graphics}
\usepackage[dvips]{color}
\usepackage{graphicx}
\usepackage[english]{babel}
\usepackage[utf8]{inputenc}
\usepackage{comment}
\usepackage{todonotes}
\usepackage{esint}
\usepackage{stmaryrd}
\usepackage{accents}

\usepackage[colorlinks,pdfpagelabels,pdfstartview = FitH,bookmarksopen
= true,bookmarksnumbered = true,linkcolor = blue,plainpages =
false,hypertexnames = false,citecolor = red,pagebackref=false]{hyperref}

\allowdisplaybreaks
\renewcommand{\epsilon}{\varepsilon}

\newtheorem{theorem}{Theorem}[section]
\newtheorem{lemma}[theorem]{Lemma}

\theoremstyle{definition}
\newtheorem{definition}[theorem]{Definition}

\newtheorem{remark}[theorem]{Remark}

\numberwithin{equation}{section}

\title[Regularity of solutions and supersolutions]{Regularity of weak solutions and supersolutions to the porous medium equation}

\author[V. B\"ogelein]{Verena B\"{o}gelein}
\address{Verena B\"ogelein\\
Fachbereich Mathematik, Universit\"at Salzburg\\
Hellbrunner Str. 34, 5020 Salzburg, Austria}
\email{verena.boegelein@sbg.ac.at}

\author[P. Lehtel\"a]{Pekka Lehtel\"a}
\address{Pekka Lehtel\"a\\
Aalto University, Department of Mathematics and Systems Analysis\\
P.O. Box 11100, FI-00076 Aalto, Finland}
\email{pekka.lehtela@aalto.fi}

\author[S. Sturm]{Stefan Sturm}
\address{Stefan Sturm\\
Fachbereich Mathematik, Universit\"at Salzburg\\
Hellbrunner Str. 34, 5020 Salzburg, Austria}
\email{stefan.sturm@sbg.ac.at}

\thanks{The research is partially supported by the Emil Aaltonen Foundation.}

\newcommand\rn{\mathbb R^n}
\newcommand\re{\mathbb R}
\newcommand\n{\mathbb N}

\newcommand\bd{\partial}
\newcommand\ph{\varphi}
\newcommand\eps{\varepsilon}

\renewcommand\div{\operatorname{div}}
\newcommand\diam{\operatorname{diam}}
\newcommand\dist{\operatorname{dist}}

\DeclareMathOperator*{\esssup}{ess\,sup}

\newcommand\dx {\, d}

\DeclareRobustCommand{\rchi}{{\mathpalette\irchi\relax}} 
\newcommand{\irchi}[2]{\raisebox{\depth}{$#1\chi$}} 

\providecommand{\ch}[1]{\text{\raise 2pt \hbox{$\chi$}\kern-0.2pt}_{#1}}

\providecommand{\vint}[1]{\mathchoice
          {\mathop{\vrule width 5pt height 3 pt depth -2.5pt
                  \kern -9pt \kern 1pt\intop}\nolimits_{\kern -5pt{#1}}}%
          {\mathop{\vrule width 5pt height 3 pt depth -2.6pt
                  \kern -6pt \intop}\nolimits_{\kern -3pt{#1}}}%
          {\mathop{\vrule width 5pt height 3 pt depth -2.6pt
                  \kern -6pt \intop}\nolimits_{\kern -3pt{#1}}}%
          {\mathop{\vrule width 5pt height 3 pt depth -2.6pt
                  \kern -6pt \intop}\nolimits_{\kern -3pt{#1}}}}

\begin{document}

\begin{abstract}
We study the relations between different regularity assumptions in the definition of weak solutions and supersolutions to the porous medium equation. In particular, we establish the equivalence of the conditions $u^m \in L^2_\textup{loc}(0,T;H^1_\textup{loc}(\Omega))$ and $u^\frac{m+1}2 \in L^2_\textup{loc}(0,T;H^1_\textup{loc}(\Omega))$ in the definition of weak solutions. Our proof is based on approximation by solutions to obstacle problems.
\end{abstract}



\date{\today}  
\maketitle

\section{Introduction}
In this paper, we study the connections between various notions of non-negative (super)solutions to the slow diffusion porous medium equation
\begin{equation}
  \label{eq:PME}
  \partial_t u-\Delta u^m= 0 \quad \text{in } \Omega_T,
\end{equation}
where $m>1$ and $\Omega_T= \Omega\times (0,T)$ with $T>0$ denotes the space-time cylinder over a bounded domain $\Omega\subset \rn$. Particularly, concerning the notion of weak solution in the case $m>1$, there are basically two different definitions used in the literature. The difference becomes apparent in the regularity assumptions utilized in the definition of weak solutions. 
The first one, which is used for instance in \cite{BDKS, DiBenedetto_Holder, supersol, vazquez}, acts on the assumptions that  
\begin{equation*}
	u\in C^0\big((0,T);L^{m+1}_{\rm loc}(\Omega)\big)
	\quad\text{and}\quad
	u^m \in L^2_{\textup{loc}}\big(0,T; H^1_{\textup{loc}}(\Omega)\big),
\end{equation*}
whereas the requirements in the other definition of weak solutions, used for instance in \cite{PME_measuredata, DBGV-Acta, dibenedetto_harnack, gianazza_schwarzacher}, are
\begin{equation*}
	u\in C^0\big((0,T);L^{2}_{\rm loc}(\Omega)\big)
	\quad\text{and}\quad
	u^{\frac{m+1}{2}} \in L^2_{\textup{loc}}\big(0,T; H^1_{\textup{loc}}(\Omega)\big). 
\end{equation*}
As the relations between these conditions have, at least to the authors' knowledge, not been treated up to now, we aim to clarify this matter by establishing the equivalence of these two conditions in the definition of weak solutions; see Theorem~\ref{equivalence_celebrated} below. 

As a tool, which is interesting also in its own right, we consider a class of supersolutions that are defined analogously to supercaloric functions in classical potential theory, i.\,e.\ in terms of the parabolic comparison principle. Originally, this class was introduced as ``viscosity supersolutions'' in \cite{supersol}, and recently, the label ``$m$-superporous functions'' was suggested in \cite{kinnunnen_lindqvist_unbounded, lehtela_lukkari}. As it seems to be more natural, we will call them {\it $m$-supercaloric functions} (see Def.\ \ref{def:supercaloric}). Apart from that, another notion is the one of {\it weak supersolutions} (see Def.\ \ref{def:supersolutions}), which are defined via the weak formulation of \eqref{eq:PME}. While $m$-supercaloric functions typically appear in potential theory, the weak formulation provides a natural approach to regularity questions. 

In \cite{supersol}, it was proved that these two notions are related in the sense that bounded $m$-supercaloric functions are also weak supersolutions to the porous medium equation. According to that result, $m$-supercaloric functions can be studied via approximation by their truncations, which are weak supersolutions. However, the defintion of weak supersolutions used for example in \cite{ivanov, supersol} does not allow to work with the solution itself as a test function. Instead, only $u^m$ is an admissible choice. Therefore, some additional assumptions are needed when testing the equation with $u$. A sufficient condition, that is commonly imposed (see \cite{PME_measuredata, DBGV-Acta, dibenedetto_harnack, gianazza_schwarzacher}), is given by the integrability property \eqref{integrab_cond}. 

Our first theorem shows that this property is satisfied for bounded $m$-supercaloric functions. In a sense, it complements the work of \cite{supersol} as it allows to study the regularity of $m$-supercaloric functions by applying the results for weak supersolutions which were established under the assumption \eqref{integrab_cond}. 

\begin{theorem} \label{thm:equivalence}
  Let $u$ be a locally bounded $m$-supercaloric function. Then, $u$ is a weak supersolution to the porous medium equation \eqref{eq:PME} and $u$ satisfies
\begin{align} \label{integrab_cond}
u^{\frac{m+1}2} \in L^2_\textup{loc}\big(0,T;H^1_\textup{loc}(\Omega)\big).
\end{align}
\end{theorem}


The novelty of this paper is proving the existence of the gradient $\nabla u^\frac{m+1}{2}$ in $L^2_\textup{loc}(\Omega_T,\rn)$ for locally bounded $m$-supercaloric functions $u$. Moreover, we rigorously establish all the tools concerning solutions to the obstacle problem used throughout the proof, like the assertion that weak solutions to the obstacle problem are weak solutions to the porous medium equation in the complement of the coincidence set (see Lemma~\ref{weak_sol_exist_lemma}). 
This closes a gap in the literature, since such assertions were commonly recognized to be true although to our knowledge, a rigorous proof was missing. 

From the above theorem, we can also infer the equivalence of locally bounded $m$-supercaloric functions and weak supersolutions satisfying \eqref{integrab_cond} in the following way. First, our theorem shows that locally bounded $m$-supercaloric functions are weak supersolutions fulfilling the conditions $u^m \in L^2_{\textup{loc}}(0,T; H^1_{\textup{loc}}(\Omega))$ and $u^{\frac{m+1}2} \in L^2_{\textup{loc}}(0,T;H^1_{\textup{loc}}(\Omega))$. On the other hand, the fact that weak supersolutions are $m$-supercaloric functions follows directly from the comparison principle (see \cite[Thm.\ 6.5]{vazquez}) and the existence of lower semicontinuous representatives (see \cite[Thm.\ 1.1]{lower_semicontinuity}). 

The analogue result for the evolutionary $p$-Laplace equation was established in \cite{kinnunen_lindqvist_reg}, where the counterpart of $m$-supercaloric functions are $p$-superparabolic functions. In other words, there was shown that bounded $p$-superparabolic functions are weak supersolutions to the evolutionary $p$-Laplace equation, and the other implication follows in the same manner as for the porous medium equation. However, despite the similarities of both equations, the necessity of a condition like \eqref{integrab_cond} is typical for the porous medium equation whereas such a phenomenon does not appear in the theory of the $p$-Laplace equation.

Coming back to the porous medium equation, similar challenges as for weak supersolutions also arise when trying to define a suitable notion of weak solution in order to establish their regularity. However, as weak solutions possess a continuous representative by \cite[Thm.\ 1.1]{dahlberg_kenig}, we can assume that they are locally bounded (see also \cite{bound_andreucci} for an explicit estimate), and since also the comparison principle is at hand, Theorem~\ref{thm:equivalence} is applicable. This shows that, for weak solutions, $u^m \in L^2_\textup{loc}(0,T;H^1_\textup{loc}(\Omega))$ implies $u^\frac{m+1}2 \in L^2_\textup{loc}(0,T;H^1_\textup{loc}(\Omega))$. On the one hand, this result can be seen as a direct consequence of Theorem~\ref{thm:equivalence}. On the other hand, it can be proved directly, without employing the theory of $m$-supercaloric functions. Since also this viewpoint is interesting -- in particular for more general porous medium type equations -- we provide this alternative proof at the end of Section~\ref{sec:proof}. In addition, we will even prove the converse statement, which ensures that we will get the following equivalence of the two notions of weak solutions to the porous medium equation mentioned above.

\begin{theorem} \label{equivalence_celebrated}
For any function $u\colon \Omega_T\to [0,\infty]$, the following statements are equivalent:
\begin{enumerate}
\item[(i)] $u\in C^0((0,T);L^{m+1}_{\rm loc}(\Omega))$ and $u^m \in L^2_\textup{loc}(0,T;H^1_\textup{loc}(\Omega))$ and
	\begin{align}\label{weak-eq}
		\iint_{\Omega_T} \Big(-u\partial_t\ph + \nabla u^m \cdot \nabla \ph \Big) \, dx \, dt = 0	
	\end{align}
	for any test function $\ph\in C_0^\infty(\Omega_T)$;\vspace{.1cm}
\item[(ii)] $u\in C^0((0,T);L^{2}_{\rm loc}(\Omega))$ and $u^\frac{m+1}{2} \in L^2_\textup{loc}(0,T;H^1_\textup{loc}(\Omega))$ and
	\begin{align}\label{weak-eq-it}
	   \iint_{\Omega_T} \Big(-u\partial_t\ph + \tfrac{2m}{m+1} u^\frac{m-1}{2} \nabla u^\frac{m+1}{2} \cdot \nabla \ph \Big) \, dx \, dt = 0
	\end{align}
	for any test function $\ph\in C_0^\infty(\Omega_T)$.
\end{enumerate}
\end{theorem}

\begin{remark}
The assertion (i) provides the standard weak formulation coming from multiplication by a test function and integration by parts in \eqref{eq:PME}. By contrast, under the assumption $u^\frac{m+1}{2} \in L^2_\textup{loc}(0,T;H^1_\textup{loc}(\Omega))$, the weak formulation of \eqref{eq:PME} can be defined as in (ii) if we understand the gradient $\nabla u^{\frac{m+1}{2}}$ in the sense
\begin{align*}
\nabla u^\frac{m+1}{2} = \tfrac{m+1}{2m} \rchi_{\{u>0\}} u^\frac{1-m}{2} \nabla u^m.
\end{align*}
Note that this interpretation is in accordance with the definition of weak solutions given, for instance, in \cite[Def.\ 1.1]{PME_measuredata} and \cite[eq.\ (5.7)]{ivanov}.
\end{remark}

To conclude the introduction, we will present the basic ideas of the proof of Theorem \ref{thm:equivalence}, which will imply Theorem \ref{equivalence_celebrated} as described. Since an $m$-supercaloric function $u$ is lower semicontinuous, it can be approximated by an increasing sequence of smooth functions $\psi_i$. The idea is to consider solutions $u_i$ to the obstacle problem with $\psi_i$ as an obstacle. Then, it can be shown that $u_i$ are weak supersolutions and $u_i\rightarrow u$ in a suitable sense. Our contribution here is providing uniform $L_\textup{loc}^2(\Omega_T,\rn)$-estimates for the gradients $\nabla u_i^{\frac{m+1}2}$ of the weak solutions to the obstacle problem, which ensure that $\nabla u^{\frac{m+1}2}\in L^2_{\rm loc}(\Omega_T,\rn)$ after passing to the limit $i\to\infty$. Moreover, as mentioned before, we provide rigorous proofs for some tools on obstacle problems which were commonly recognized to be true, like the assertion on the complement of the coincidence set in Lemma~\ref{weak_sol_exist_lemma}. In a certain sense, this complements the work of \cite{obstacle}. More precisely, in Sections 3 and 4, we modify their approach of constructing weak solutions to the obstacle problem by establishing the estimates for $\nabla u_i^\frac{m+1}{2}$ which were not taken into account in \cite{obstacle}.  
Finally, via the described estimates and weak compactness, we can conclude that \eqref{integrab_cond} is valid.

\section{Preliminaries}\label{sec:pre}

To start with, we fix some notations. We let $\Omega$ be a bounded domain in $\rn$ and denote by $\Omega_T:=\Omega\times(0,T)$ the space-time cylinder of height $T>0$ over $\Omega$. For $U\Subset\Omega$ and $0<t_1<t_2<T$, we abbreviate the cylinder $U\times (t_1,t_2)$ by $U_{t_1,t_2}$ and its parabolic boundary by $\bd_p U_{t_1,t_2} := \big(\bd U\times (t_1,t_2)\big) \cup \big(\overline U \times \{t_1\}\big)$. For an open ball of radius $\varrho>0$ centered at $x_0\in\Omega$, we write $B(x_0,\varrho)$. Moreover, we denote the positive and negative parts of a function $u$ by $u_+ := \max\{u,0\}$ and $u_- := \max\{-u,0\}$, respectively, and, $C$ stands for a constant, which may vary from line to line.

Next, we define our notion of weak (super)solutions.

\begin{definition} \label{def:supersolutions}
  A non-negative $u\in C^0((0,T);L^{m+1}_{\rm loc}(\Omega))$ is a {\it weak supersolution} to the porous medium equation \eqref{eq:PME} if $u^m \in L^2_{\textup{loc}}(0,T; H^1_{\textup{loc}}(\Omega))$ and $u$ satisfies
\begin{equation}\label{weak-super}
\iint_{\Omega_T} \Big(-u\partial_t\ph + \nabla u^m \cdot \nabla \ph \Big) \, dx \, dt \ge 0
\end{equation}
for any test function $\ph \in C_0^\infty (\Omega_T)$ with $\ph \ge 0$. Similarly, $u$ is a {\it weak subsolution} if the above inequality holds reversed. Moreover, $u$ is a {\it weak solution} if it is a weak sub- and supersolution.
\end{definition}

We continue by giving the definition of $m$-supercaloric functions. 

\begin{definition} \label{def:supercaloric}
  A function $u\colon\Omega_T\rightarrow [0,\infty]$ is {\it $m$-supercaloric} if 
  \begin{enumerate}
  \item $u$ is lower semicontinuous,
  \item $u$ is finite in a dense subset of $\Omega_T$, and
  \item $u$ satisfies the following comparison principle in every interior cylinder $U_{t_1,t_2}\Subset \Omega_T$: 
If $w\in C^0(\overline{U_{t_1,t_2}})$ is a weak solution to \eqref{eq:PME} in $U_{t_1,t_2}$ and $u\ge w$ on $\bd_p U_{t_1,t_2}$, then $u\ge w$ in $U_{t_1,t_2}$.
  \end{enumerate}
\end{definition}

By \cite[Thm.\ 1.3]{supersol}, bounded $m$-supercaloric functions are weak supersolutions in the sense of Def.\ \ref{def:supersolutions}.

In order to derive energy estimates for weak (super)solutions to the porous medium equation, we want to use the (super)solution $u$ itself as a test function. However, since the time derivative $\partial_t u$ does not exist in general, we need to regularize $u$ to deal with this matter. To that end, we introduce the mollification
\begin{equation}\label{regularization}
\llbracket u \rrbracket_h(x,t)= e^{-\frac th}v_0 + \frac 1h \int_0^t e^{\frac{s-t}h} u(x,s) \, ds
\end{equation}
for $h>0$ with some $v_0\in L^1(\Omega)$. Then, choosing $v_0\equiv 0$, it can be shown that $\llbracket u \rrbracket_h$ satisfies the regularized inequality (see \cite[eq.\ (6.5)]{obstacle} or \cite[eq.\ (2.12)]{supersol})
\begin{equation} \label{ineq:regularized}
\iint_{\Omega_T} \Big( \bd_t \llbracket u \rrbracket_h  \ph + \nabla \llbracket u^m \rrbracket_h \cdot \nabla \ph \Big)\, dx \,dt \geq  \frac 1 h \int_\Omega u(\cdot,0) \int_0^T \ph e^{-\frac s h } \, ds \, dx
\end{equation}
for any test function $\ph\in L^2(0,T;H_0^1(\Omega))$. Moreover, the formula
\begin{align} \label{magic_formula_mollification}
\partial_t \llbracket u \rrbracket_h = \tfrac 1h \big(u- \llbracket u \rrbracket_h\big)
\end{align}
holds (see \cite[Lemma 3.1]{obstacle}). By imposing the additional regularity condition \eqref{integrab_cond} to a weak supersolution $u$, we can now derive the following Caccioppoli-type estimate.

\begin{lemma}\label{caccioppoli}
  Let $u$ be a weak supersolution to the porous medium equation in the sense of Def.\ \ref{def:supersolutions} which additionally satisfies $u^{\frac {m+1}2} \in L^2_{\textup{loc}}(0,T; H^1_{\textup{loc}}(\Omega))$, and let $\zeta\in C_0^\infty(\Omega_T)$ be a smooth cut-off function such that $0\le \zeta \le 1$. Then, if $u\leq M$ in $\Omega_T$ for some constant $M>0$, the Caccioppoli estimate
  \begin{align*}
    &\iint_{\Omega_T} \zeta^2 \big| \nabla u^\frac{m+1}{2} \big|^2 \dx x \dx t \\
&~~\le C \left ( \iint_{\Omega_T} (M-u)^2 \zeta |\partial_t\zeta| \dx x \dx t + \iint_{\Omega_T} u^{m-1} (M-u)^2 |\nabla \zeta |^2 \dx x \dx t\right )
  \end{align*}
holds with a constant $C= C(m)$.
\end{lemma}

\begin{proof}
For $\epsilon\in(0,M)$, we let $g_\epsilon(s):=M-\max\{\epsilon,s\}$ for any $s\geq 0$. Moreover, we recall Definition \eqref{regularization} of the mollification in time, where throughout this proof we choose $v_0=0$ as initial value. In the regularized inequality \eqref{ineq:regularized}, we insert the test function $\varphi = \zeta^2 g_\epsilon(u)$. Note that $\varphi$ is admissible since $g_\epsilon(u)\in L^2_{\rm loc}(0,T;W^{1,2}_{\rm loc}(\Omega))$. We first treat the evolutionary integral. Using formula \eqref{magic_formula_mollification} for the time derivative of the mollification, integrating by parts and then passing in turn to the limits $h\to 0$ and $\epsilon\to 0$, we find
\begin{align*}
\iint_{\Omega_T} & \zeta^2  \bd_t \llbracket u \rrbracket_h \;\! g_\epsilon(u)\,dx\,dt \\
&= 
\iint_{\Omega_T}  \zeta^2  \bd_t \llbracket u \rrbracket_h \;\! g_\epsilon\big(\llbracket u \rrbracket_h\big)\,dx\,dt +
\iint_{\Omega_T}  \zeta^2  \bd_t \llbracket u \rrbracket_h \;\! 
\Big(g_\epsilon(u) - g_\epsilon\big(\llbracket u \rrbracket_h\big) \Big)\,dx\,dt \\
&\le 
\iint_{\Omega_T}  \zeta^2  \;\!
\bd_t \bigg(\int_M^{\llbracket u \rrbracket_h} g_\epsilon(s)\,ds\bigg)\,dx\,dt 
= 
-2\iint_{\Omega_T} \zeta\partial_t\zeta\;\! \int_M^{\llbracket u \rrbracket_h} g_\epsilon(s)\,ds \,dx\,dt \\
&\to 
-2\iint_{\Omega_T} \zeta\partial_t\zeta\;\! \int_M^u (M-s)\,ds \,dx\,dt
=
\iint_{\Omega_T} \zeta\partial_t\zeta\;\! (M-u)^2 \,dx\,dt.
\end{align*}
Note that, in the limit $h\to 0$, the right-hand side term in \eqref{ineq:regularized} vanishes, and the diffusion term reads as
\begin{align*}
& 2\iint_{\Omega_T}  \zeta\;\! g_\epsilon(u) \nabla u^m \cdot \nabla\zeta\,dx\,dt -  \iint_{\Omega_T\cap\{u>\epsilon\}} \zeta^2\;\! \nabla u^m \cdot\nabla u \,dx\,dt \\
&=
\tfrac{4m}{m+1} \iint_{\Omega_T} \zeta g_\epsilon(u)u^{\frac{m-1}{2}} \nabla u^{\frac{m+1}{2}} \cdot \nabla\zeta\,dx\,dt - 
\tfrac{4m}{(m+1)^2} \iint_{\Omega_T\cap\{u>\epsilon\}}\! \zeta^2 \big| \nabla u^\frac{m+1}{2} \big|^2 dx\,dt .
\end{align*}
Finally, after letting $\epsilon\to 0$, an application of Young's inequality shows that this expression can be bounded from above by
\begin{align*}
2m \iint_{\Omega_T} u^{m-1} (M-u)^2 |\nabla \zeta|^2\,dx\,dt - \tfrac{2m}{(m+1)^2} \iint_{\Omega_T} \zeta^2 \big| \nabla u^\frac{m+1}{2} \big|^2 \,dx\,dt,
\end{align*}
which proves the claim.
\end{proof}

As the proof of Theorem \ref{thm:equivalence} is based on approximations by solutions to obstacle problems for the porous medium equation, we introduce the concept of obstacle problems now. The idea is to find a function $u$ lying above a given obstacle function $\psi$ and attaining fixed boundary and initial values $g$ and $u_0$. In addition to that, $u$ needs to fulfill a variational inequality. More precisely, formally $u$ is required to satisfy
\begin{equation}\label{obstacle inequality}
\begin{cases}
\displaystyle\iint_{\Omega_T} \Big[\partial_t u \big(v^m-u^m\big)+ \nabla u^m \cdot \big(\nabla v^m - \nabla u^m\big) \Big]\, dx \, dt \ge 0, \\
u\ge \psi ~~\text{a.\,e.\ in } \Omega_T, \\
u=g ~~\text{on } \partial\Omega\times (0,T), \\
u(\cdot,0)=u_0 ~~\text{in } \Omega
\end{cases}
\end{equation}
for all comparison maps $v\ge \psi$ which take the same boundary and initial values as $u$. For our purposes, it suffices to consider boundary and initial values given by the obstacle function $\psi$. However, the following estimates hold also in the case of sufficiently regular data $g$ and $u_0$ as long as they are bounded. 

We consider two classes of solutions to the obstacle problem, namely strong and weak solutions (see Def.\ \ref{def:strong_sol} and Def.\ \ref{def:weak_sol}). Since we are interested in approximating bounded $m$-supercaloric functions by solutions to obstacle problems, it is reasonable to assume that the obstacles $\psi$ are smooth and bounded, which guarantees the existence of weak solutions to the obstacle problem. If, in addition, the function $\partial_t\psi  -\Delta \psi^m$ is bounded, we have strong solutions to the obstacle problem. The main difference between both notions is that $\partial_t u$ exists in a distributional sense only for strong solutions.

The reader should be aware that the proof of Theorem \ref{thm:equivalence} has to be performed locally as in \cite[Thm.\ 3.2]{supersol}, meaning that we need to work in a cylinder which is compactly contained in $\Omega_T$. However, for the sake of a friendly notation, in Sections~\ref{sec:pre}--\ref{sec:grad}, we will work in the whole domain $\Omega_T$, having in mind that the explicit argumentation in Section~\ref{sec:proof} will be done in a smaller cylinder. 

Before we give the rigorous definition of strong solutions, we specify the regularity assumptions for the obstacle and the boundary and initial values. First, we want all data to be non-negative and bounded in the sense that
\begin{equation} \label{cond:boundedness}
\psi,g \in L^\infty(\Omega_T) ~~\text{ and }~~ u_0 \in L^\infty(\Omega).
\end{equation}
Next, we impose the compatibility conditions
\begin{equation} \label{cond:compatibility}
\begin{cases}
g \ge \psi ~~\text{a.\,e.\ in } \Omega_T,\\
u_0 \ge \psi(\cdot,0) ~~\text{a.\,e.\ in } \Omega,\\
g(\cdot,0)=u_0 ~~\text{a.\,e.\ in } \Omega,
\end{cases}
\end{equation}
and the following integrability assumptions
\begin{equation} \label{cond:integrability}
\begin{cases}
\psi^m \in L^2(0,T; H^1(\Omega)),\ \partial_t \psi^m \in L^\frac{m+1}{m}(\Omega_T),\ \psi^m(\cdot,0)\in H^1(\Omega),\\
g^m \in L^2(0,T; H^1(\Omega)),\ \partial_t g^m \in L^\frac{m+1}{m}(\Omega_T),\\
u_0\in H^1(\Omega).
\end{cases}
\end{equation}
Note that the assumptions \eqref{cond:integrability} on $\psi$ and $g$ imply $\psi^m,g^m\in C^0([0,T]; L^\frac{m+1}{m}(\Omega))$, and consequently $\psi,g\in C^0([0,T]; L^{m+1}(\Omega))$ since $m>1$. Finally, in order to prove the existence of strong solutions, we need the extra condition
\begin{equation} \label{cond:extra_strong}
\Psi := \partial_t\psi - \Delta \psi^m \in L^\infty(\Omega_T).
\end{equation}

\begin{definition} \label{def:strong_sol}
 Let \eqref{cond:boundedness}--\eqref{cond:extra_strong} be satisfied. A non-negative function $u\in C^0([0,T];L^{m+1}(\Omega))$ is a {\it strong solution} to the obstacle problem \eqref{obstacle inequality} with boundary and initial values $g$ and $u_0$ if $u$ fulfills
\begin{equation}\label{strong sol assumptions}
\begin{cases}
 u^m\in g^m + L^2(0,T;H_0^1(\Omega)),\\
u\ge \psi ~~\text{a.\,e.\ in } \Omega_T,\\
\partial_t u \in L^2 (0,T; H^{-1}(\Omega)),\\
\end{cases}
\end{equation}
$u(\cdot,0)=u_0$ in the $H^{-1}(\Omega)$-sense and the inequality
\begin{align} \label{var_ineq_strong}
\int_0^T \big\langle \partial_t u, \alpha (v^m-u^m)\big\rangle \, dt + \iint_{\Omega_T} \alpha \nabla u^m \cdot \nabla (v^m-u^m) \, dx \,dt \ge 0
\end{align}
for all comparison maps $v$ satisfying the conditions \eqref{strong sol assumptions} and for all non-negative cut-off functions $\alpha \in W^{1,\infty}([0,T])$ with $\alpha(T)=0$, where $\langle\cdot,\cdot\rangle$ denotes the dual pairing between $H^{-1}(\Omega)$ and $H^1_0(\Omega)$.
\end{definition}

However, the smoothness and boundedness of $\psi$ are not enough to guarantee that \eqref{cond:extra_strong} holds. Thus, we also need to consider weak solutions to the obstacle problem. We remark that even though \eqref{cond:boundedness} will be satisfied throughout this paper, it is not a necessary assumption to ensure the existence of weak solutions.

\begin{definition} \label{def:weak_sol}
 Let \eqref{cond:compatibility} and \eqref{cond:integrability} be satisfied. A non-negative function $u\in C^0([0,T];L^{m+1}(\Omega))$ is a {\it weak solution} to the obstacle problem \eqref{obstacle inequality} with boundary and initial values $g$ and $u_0$ if $u$ fulfills
\begin{equation}\label{weak sol assumptions}
\begin{cases}
 u^m\in g^m + L^2(0,T;H_0^1(\Omega)),\\
u\ge \psi ~~\text{a.\,e.\ in } \Omega_T,\\
\end{cases}
\end{equation}
and the inequality
\begin{align}\label{ineq-weak}
\big\langle\hspace{-.1cm}\big\langle \partial_t u, \alpha (v^m-u^m)\big\rangle\hspace{-.1cm}\big\rangle_{u_0} + \iint_{\Omega_T} \alpha \nabla u^m \cdot \nabla (v^m-u^m) \, dx \,dt \ge 0
\end{align}
for all comparison maps $v$ satisfying the conditions \eqref{weak sol assumptions} and $\partial_t v^m \in L^{\frac {m+1}m}(\Omega_T)$, and for all non-negative cut-off functions $\alpha \in W^{1,\infty}([0,T])$ with $\alpha(T)=0$, where we have denoted 
\begin{align*}
 \big\langle\hspace{-.1cm}\big\langle \partial_t u, \alpha (v^m-u^m)\big\rangle\hspace{-.1cm}\big\rangle_{u_0} &:= \iint_{\Omega_T} \Big( \alpha' \left( \tfrac 1 {m+1} u^{m+1}-uv^m\right) - \alpha u \partial_t v^m \Big) \, dx \, dt \\
&~~+ \alpha(0) \int_\Omega \left ( \tfrac 1 {m+1} u_0^{m+1} - u_0 v^m(\cdot,0) \right)\, dx.
\end{align*}
\end{definition}

Note that the initial condition $u(\cdot,0)=u_0$ is incorporated in the variational inequality \eqref{ineq-weak}; see \cite[Lemma 5.2]{obstacle}. 

\begin{remark} \label{rem:sol_supersol}
The strong and weak solutions to the obstacle problem constructed in Sections~\ref{sec:strong} and~\ref{sec:grad} are also weak supersolutions to the porous medium equation (see \cite[Thms.\ 2.6, 2.7]{obstacle}). Thus, energy estimates for weak supersolutions such as Lemma \ref{caccioppoli} hold, provided that the regularity assumption \eqref{integrab_cond} is satisfied.
\end{remark}

Finally, we cite the following parabolic Sobolev's inequality from \cite[Prop.\ I.3.1]{dibenedetto1993} (see also \cite[Lemma 2.1]{PME_measuredata}).

\begin{lemma}\label{sobolev}
  Let $B(x_0,\varrho)\subset\Omega$ and $0<t_1<t_2<T$. If
\begin{align*}
v\in L^\infty\big( t_1,t_2; L^r(B(x_0, \varrho)) \big) \cap L^p\big( t_1,t_2; W^{1,p}(B(x_0, \varrho)) \big)
\end{align*} 
with $p\in (1,\infty)$ and $r\in [1,\infty)$, there exists a constant $C= C(n,p,r)$ such that 
\begin{align*}
   &\int_{t_1}^{t_2}\int_{B(x_0,\varrho)} |v|^\ell \, dx \, dt \\ &~~\le C \int_{t_1}^{t_2}\int_{B(x_0,\varrho)} \left(\Big|\frac{v}{\varrho}\Big|^p + |\nabla v|^p \right) \, dx \, dt \left( \esssup_{t\in (t_1,t_2)} \int_{B(x_0,\varrho)} |v|^r \, dx \right)^{p/n},
  \end{align*}
where $\ell = p \frac {n+r}n$.
\end{lemma}

\section{Gradient estimates for strong solutions}\label{sec:strong}

In this section, we deal with strong solutions to the obstacle problem and assume that \eqref{cond:boundedness}--\eqref{cond:extra_strong} hold. We will show the existence of strong solutions to the obstacle problem by approximating them by weak solutions to the penalized porous medium equation
\begin{equation}
  \label{eq:penalizedPME}
  \begin{cases}
    \partial_t u - \Delta u^m = \Psi_+ \xi_\delta (\psi^m - u^m) ~~\text{in } \Omega_T, \\
u=g ~~\text{on } \bd \Omega \times (0,T),\\
u(\cdot,0)=u_0 ~~\text{in }\Omega.
  \end{cases}
\end{equation}
Here, for $\delta>0$, the function $\xi_\delta : \re \rightarrow [0,1]$ is smooth and satisfies
\[
\begin{cases}
  \xi_\delta  = 1 \quad \text{in } [0,\infty),\\
\xi_\delta = 0 \quad \text{in } (-\infty, -\delta],\\
|\xi'_\delta| \le \frac 2 \delta.
\end{cases}
\]
Observe that, whenever $u^m\ge\psi^m+\delta$, the term $\xi_\delta(\psi^m - u^m)$ vanishes and thus, \eqref{eq:penalizedPME} reduces to the ordinary initial-boundary value problem for the porous medium equation. Next, we give the rigorous definition for weak solutions to the above penalized equation.

\begin{definition}
A non-negative function $u\in C^0([0,T]; L^{m+1}(\Omega))$ is a weak solution to the penalized porous medium equation \eqref{eq:penalizedPME} if $u$ fulfills
\begin{align*}
\begin{cases}
 u^m\in g^m+ L^2(0,T;H_0^1(\Omega)),\\
u(\cdot,0)=u_0 ~~\text{in } \Omega,
\end{cases}
\end{align*}
and the equation
\begin{equation} \label{weak_penalized_PME}
\iint_{\Omega_T} \Big(-u\partial_t\ph + \nabla u^m \cdot \nabla \ph \Big) \, dx \, dt = \iint_{\Omega_T} \Psi_+ \xi_\delta (\psi^m - u^m)\ph \, dx \, dt
\end{equation}
for any test function $\ph\in C_0^\infty(\Omega_T)$.
\end{definition}

The associated averaged equation can be deduced as in \cite[eq.\ (6.5)]{obstacle} and is given by
\begin{equation} \label{eq:regularized}
\begin{aligned}
&\iint_{\Omega_T} \Big( \bd_t \llbracket u \rrbracket_h  \ph + \nabla \llbracket u^m \rrbracket_h \cdot \nabla \ph \Big)\, dx \,dt \\ &~~= \iint_{\Omega_T} \llbracket \xi_\delta (\psi^m-u^m) \Psi_+ \rrbracket_h \ph \, dx \, dt + \frac 1 h \int_\Omega u_0 \int_0^T \ph e^{-\frac s h } \, ds \, dx
\end{aligned}
\end{equation}
for any test function $\ph\in L^2(0,T;H_0^1(\Omega))$. As before, $\llbracket\cdot\rrbracket_h$ is defined according to \eqref{regularization} with $v_0=0$. We cite the following energy estimates for weak solutions to the penalized porous medium equation from \cite[Lemma 7.2]{obstacle}.

\begin{lemma}\label{penalized energy estimates}
  Let $u$ be a weak solution to \eqref{eq:penalizedPME}. Then, we have
\[
\sup_{t\in [0,T]}\int_{\Omega} u(\cdot, t)^{m+1} \, dx + \iint_{\Omega_T} |\nabla u^m|^2 \, dx \, dt \le C A
\]
and
\[
\| \partial_t u \|^2_{L^2(0,T; H^{-1}(\Omega))} \le C A.
\]
Here, $C$ is a constant depending on $n, m, \diam(\Omega)$ and $T$, and $A$ is defined as 
\begin{equation}
\begin{aligned} \label{def:constantA}
A &= \sup_{t\in [0,T]} \int_\Omega g(\cdot,t)^{m+1} \, dx + \int_\Omega u_0^{m+1} \, dx \\&~~~+ \iint_{\Omega_T} \Big ( |\Psi_+|^2 + |\nabla g^m|^2 + |\partial_t g^m|^{\frac {m+1}m} \Big) \, dx \, dt.
\end{aligned}
\end{equation}
\end{lemma}
The proof follows from a formal calculation by inserting the test function $\ph= \rchi_{[0,\tau]} (u^m - g^m)$ for some $\tau\in (0,T]$ in \eqref{eq:regularized}, where $\rchi_{[0,\tau]}$ denotes the characteristic function of the interval $[0,\tau]$. For the details, we refer to \cite{obstacle}.

Next, we will prove the following gradient estimate.

\begin{lemma}\label{penalized grad}
  Let $u$ be a weak solution to the penalized porous medium equation \eqref{eq:penalizedPME} and suppose that $u\in L^2_{\textup{loc}} ( 0,T ; H^1_{\textup{loc}}(\Omega))$. Then, we have $\nabla u^\frac{m+1}{2} \in L^2_\textup{loc}(\Omega_T,\rn)$, together with the estimate
\begin{align*}
  \iint_{U_{t_1,t_2}} \big|\nabla u^{\frac{m+1}2} \big|^2 \, dx \, dt \le C(A+1)
\end{align*}
for any $U\Subset\Omega$ and any $0<t_1<t_2<T$ with $A$ as in \eqref{def:constantA} and a constant $C=C(n,m,U,t_1,t_2,\Omega,T)$.
\end{lemma}

\begin{proof}
We fix $U,t_1,t_2$ as in the statement of the lemma. To show that $\nabla u^\frac{m+1}{2}$ exists, we define $u_k = \min \{ u,k \}$ for $k\in\n$, and note that $[0,k]\ni s\mapsto s^\frac{m+1}{2}$ is Lipschitz continuous so that $u_k^\frac{m+1}{2}$ is weakly differentiable with $$\nabla u_k^\frac{m+1}{2}=\tfrac{m+1}{2}u_k^\frac{m-1}{2} \nabla u_k.$$ More precisely, for any $i\in\{1,\dots,n\}$, we have that 
\begin{align*}
\iint_{U_{t_1,t_2}} u_k^\frac{m+1}{2} \partial_{x_i} \varphi \,dx\,dt 
= 
-\tfrac{m+1}{2} \iint_{U_{t_1,t_2}} u_k^\frac{m-1}{2} \partial_{x_i} u_k\, \varphi \,dx\,dt
\end{align*}
for any $\varphi\in C_0^\infty(U_{t_1,t_2})$. As the right-hand side remains bounded in the limit $k\to\infty$, the gradient $\nabla u^\frac{m+1}{2}$ exists and is given by $\nabla u^\frac{m+1}{2} = \frac{m+1}{2} u^\frac{m-1}{2} \nabla u$. Next, in order to prove the gradient bound, we choose a non-negative cut-off function $\zeta\in C_0^\infty (\Omega_T)$ with $\zeta=1$ in $U_{t_1,t_2}$ and $|\nabla \zeta| + |\partial_t\zeta| \leq \frac{C}{\dist(U_{t_1,t_2},\,\partial\Omega_T)}$. Then, we insert the test function $\ph= \zeta^2 u$ in the regularized equation \eqref{eq:regularized}. We begin by treating the term involving the time derivative. Observing that, by \eqref{magic_formula_mollification}, we have 
\begin{align*}
\iint_{\Omega_T} \zeta^2 u \bd_t  \llbracket u \rrbracket_h  \, dx \, dt &= \iint_{\Omega_T} \zeta^2 \llbracket u \rrbracket_h \bd_t  \llbracket u \rrbracket_h  \, dx \, dt + \iint_{\Omega_T} \zeta^2 \big(u - \llbracket u \rrbracket_h\big) \bd_t  \llbracket u \rrbracket_h \, dx \, dt \\ &\ge  \tfrac 12 \iint_{\Omega_T} \zeta^2 \bd_t \llbracket u \rrbracket_h ^2 \, dx \,dt = - \iint_{\Omega_T} \zeta \partial_t \zeta \llbracket u \rrbracket_h ^2 \, dx \, dt
\end{align*}
and considering \cite[Lemma 2.2]{kinnunen_lindqvist_reg} for the convergence properties of the mollification, we may let $h\rightarrow 0$. Then, inserting the above inequality in \eqref{eq:regularized}, we get
\begin{align*}
  \iint_{\Omega_T} \Big(-\zeta \partial_t \zeta u^2 + \nabla u^m \cdot \nabla (\zeta^2 u) \Big)\, dx \,dt  \le& \iint_{\Omega_T}  \zeta^2 u\, \Psi_+ \xi_\delta (\psi^m-u^m) \, dx \, dt. 
\end{align*}
By Young's inequality, we may write 
\begin{align*}
\nabla u^m \cdot \nabla (\zeta^2 u) &= \tfrac{4m}{(m+1)^2}\zeta^2 \big|\nabla u^{\frac {m+1} 2 } \big|^2 + \tfrac{4m}{m+1} \zeta u^{\frac{m+1}{2}} \nabla u^{\frac{m+1}{2}} \cdot \nabla \zeta \\ 
&\geq \tfrac{2m}{(m+1)^2}\zeta^2 \big|\nabla u^{\frac {m+1}{2} } \big|^2 - 2m u^{m+1} |\nabla \zeta|^2.
\end{align*}
Furthermore, we treat the penalty term by applying Young's inequality to $u\Psi_+$ and by using the facts that $\zeta\le 1$ and $\xi_\delta \le 1$. In this way, we find 
\begin{align*}
\iint_{\Omega_T} & \zeta^2 \big|\nabla u^{\frac{m+1}2} \big|^2 \, dx \, dt \\ &\le C \iint_{\Omega_T} \Big( u^{m+1} |\nabla \zeta|^2 + \Psi_+^2 + u^2 + \zeta|\partial_t\zeta|u^2 \Big) \, dx \,dt \\ &\leq C \iint_{\Omega_T} \Big( u^{m+1} |\nabla \zeta|^2 + \Psi_+^2 + u^{m+1} + \big( \zeta|\partial_t\zeta| \big)^\frac{m+1}{m-1} +1 \Big) \, dx \,dt, 
\end{align*}
and, employing the $L^\infty(0,T;L^{m+1}(\Omega))$-bound for $u$ from Lemma \ref{penalized energy estimates}, this inequality proves the claim.
\end{proof}

Now, we will show that there exists a weak solution to \eqref{eq:penalizedPME} satisfying the above energy estimates. Our contribution is proving that the estimate for $\nabla u^{\frac {m+1}2}$ holds, and for the reader's convenience, we present the key ideas of the existence proof from \cite[Lemma 7.3]{obstacle} as well.

\begin{lemma} \label{exist_weak_penalized}
  There exists a weak solution $u_\delta$ to \eqref{eq:penalizedPME} such that $u_\delta\ge \psi$ a.\,e.\ in $\Omega_T$. Moreover, we have $\nabla u_\delta^\frac{m+1}{2} \in L^2_\textup{loc}(\Omega_T,\rn)$, and the following estimates hold:
\begin{equation}\label{eq:energy estimates}
\sup_{t\in[0,T]} \int_\Omega u_\delta^{m+1} \, dx + \big\| \nabla u_\delta^m\big\|_{L^2(\Omega_T,\rn)} + \big\| \partial_t u_\delta \big\|_{L^2(0,T ; H^{-1}(\Omega))}\le C_1
\end{equation}
and
\begin{equation}\label{eq:energy estimates_1}
\big\|\nabla u_\delta^{\frac {m+1}2}\big\|_{L^2(U_{t_1,t_2},\rn)} \le C_2,
\end{equation}
where $C_1 = C_1(n,m,\Omega,T,A)$ and $C_2= C_2(n,m,U,t_1,t_2,\Omega,T,A)$ are constants and $A$ is as in \eqref{def:constantA}.
\end{lemma}

\begin{proof}
For $\eps, \gamma, \delta \in (0,1]$, we define 
\begin{align*}
 & \psi_\eps = \psi + \eps, \\
&g_{\eps,\gamma} = (g^m+\gamma^m)^{\frac 1 m} + \eps,\\
&u_{0,\eps,\gamma}= u_0 + \eps + \gamma,\\
&\Psi_\eps = \partial_t \psi_\eps - \Delta \psi_\eps^m.
\end{align*}
One can easily see that $g_{\eps,\gamma} \ge g+ \eps \ge \psi_\eps$ and $u_{0,\eps,\gamma} \ge \psi_{\eps}(\cdot, 0)$. Denoting
\[
N= \max \bigg\{ \sup_{\Omega_T} (\psi_\eps^m + \delta )^{\frac 1 m},\, \sup_{\Omega_T} g_{\eps,\gamma},\, \sup_{\Omega} u_{0,\eps,\gamma} \bigg\},
\]
it follows that $\psi_\eps,\, g_{\eps, \gamma},\, u_{0,\eps,\gamma} \le N$, and by choosing $\eps$ and $\gamma$ smaller if necessary, we can assure that $N\le \frac 1 {\eps+\gamma}$. Next, we define
\[
a_{\eps}(s) =
\begin{cases}
  m \eps^{m-1}, \quad 0\le s \le \eps,\\
m s^{m-1}, \quad \eps < s \le \frac 1 \eps, \\
m \eps^{1-m}, \quad s\ge \frac 1 \eps.
\end{cases}
\]
Then, by \cite[Thm.\ 1.2, p.\ 162\ f.]{Lions}, there exists a weak solution $u_{\eps,\gamma,\delta} \in C^0([0,T]; L^2(\Omega)) \cap L^2(0,T; H^1(\Omega))$ to the initial-boundary value problem 
  \[
\begin{cases}
   \bd_t u - \div \big(a_\eps (u) \nabla u\big) = (\Psi_\eps)_+ \xi_\delta (\psi_\eps^m - u^m) ~~\text{in } \Omega_T,\\
u = g_{\eps,\gamma} ~~\text{on } \bd \Omega \times (0,T),\\
u(\cdot,0)=u_{0,\eps,\gamma } ~~\text{in } \Omega.
\end{cases}  
\]
Since $u_{\eps,\gamma,\delta}$ satisfies the comparison principle (see \cite[Lemma 7.1]{obstacle}) and the constants $\eps+\gamma$ and $N$ are solutions, we have 
\[
\eps+\gamma \le u_{\eps,\gamma,\delta} \le N \le \frac 1 {\eps+\gamma} \quad \text{a.\,e.\ in } \Omega_T.
\]
Thus, we obtain $a_\eps (u_{\eps,\gamma,\delta}) \nabla u_{\eps,\gamma,\delta} = \nabla u_{\eps,\gamma,\delta}^m$, which implies that $u_{\eps,\gamma,\delta}$ is also a weak solution to \eqref{eq:penalizedPME} with boundary values $g_{\eps,\gamma}$ and initial values $u_{0,\eps,\gamma}$. Hence, by Lemma \ref{penalized energy estimates} and Lemma \ref{penalized grad}, we have the following energy estimates for $u_{\eps,\gamma,\delta}$:
\begin{equation}\label{eq:pre_energy estimates}
  \sup_{t\in[0,T]} \int_\Omega u_{\eps,\gamma,\delta}^{m+1} \, dx + \big\| \nabla u_{\eps,\gamma,\delta}^m\big\|_{L^2(\Omega_T,\rn)} +
  \big\| \partial_t u_{\eps,\gamma,\delta} \big\|_{L^2(0,T ; H^{-1}(\Omega))}\le C_1
\end{equation}
and 
\begin{equation}\label{eq:pre_energy estimates_1}
\big\|\nabla u_{\eps,\gamma,\delta}^{\frac {m+1}2}\big\|_{L^2(U_{t_1,t_2},\rn)} \le C_2
\end{equation}
for any $U\Subset\Omega$ and any $0<t_1<t_2<T$ with constants $C_1 = C_1(n,m,\Omega,T,A)$ and $C_2= C_2(n,m,U,t_1,t_2,\Omega,T,A)$, where $A$ is as in \eqref{def:constantA}. Note that $C_1$ and $C_2$ are independent of $\eps$, $\gamma$ and $\delta$. Moreover, by \cite[Thm.\ 1.2]{DiBenedetto_Holder}, the weak solutions $u_{\eps,\gamma,\delta}$ are locally H\"older continuous with an estimate independent of $\eps$. Therefore, we may let $\eps\rightarrow 0$ and, subsequently by monotone convergence, $\gamma \rightarrow 0$ to conclude that $u_{\eps,\gamma,\delta}$ subconverge to a weak solution $u_\delta$ to the penalized porous medium equation in the following sense:
$u_{\eps,\gamma,\delta}\to u_\delta$ a.\,e.\ in $\Omega_T$, $\nabla u_{\eps,\gamma,\delta}^m \rightharpoondown \nabla u_{\delta}^m$ weakly in $L^2(\Omega_T,\rn)$, $u_{\eps,\gamma,\delta}~ \accentset{\scriptstyle \ast}{\rightharpoondown}~ u_{\delta}$ weakly-$\ast$ in $L^\infty(0,T;L^{m+1}(\Omega))$, $\partial_t u_{\eps,\gamma,\delta} \rightharpoondown \partial_t u_{\delta}$ weakly in $L^2(0,T;H^{-1}(\Omega))$, and $\nabla u_{\eps,\gamma,\delta}^\frac{m+1}{2} \rightharpoondown \nabla u_{\delta}^\frac{m+1}{2}$ weakly in $L^2(U_{t_1,t_2},\rn)$. In addition, we have $u_\delta \ge \psi$, and the estimates from \eqref{eq:pre_energy estimates} and \eqref{eq:pre_energy estimates_1} persist in the limit because the upper bounds $C_1$ and $C_2$ are uniform with respect to $\varepsilon$ and $\gamma$. For the details, we refer to \cite[Prop.\ 7.3]{obstacle}.
\end{proof}

Now, we are ready to show the existence of a strong solution to the obstacle problem.

\begin{lemma}\label{strong solution lemma}
Suppose that the conditions \eqref{cond:boundedness}--\eqref{cond:extra_strong} hold. Then, there exists a strong solution $u$ to the obstacle problem \eqref{obstacle inequality}, which satisfies $u^{\frac{m+1}2}\in  L^2_{\textup{loc}}(0,T;H^1_{\textup{loc}}(\Omega))$.
\end{lemma}

\begin{proof}

By Lemma \ref{exist_weak_penalized}, there exists a weak solution $u_\delta$ to the penalized porous medium equation. From \cite[Thm.\ 1.2]{DiBenedetto_Holder}, we know that the functions $u_\delta$ are locally H\"older continuous with a quantitative estimate which is uniform in $\delta$. Therefore, there is a function $u$ such that $u_\delta\to u$ locally uniformly in $\Omega_T$ as $\delta\to 0$.  
As the energy estimates in \eqref{eq:energy estimates} and \eqref{eq:energy estimates_1} are independent of $\delta$, we find a (not relabeled) subsequence $u_\delta$ such that $\nabla u_\delta^m \rightharpoondown \nabla u^m$ weakly in $L^2(\Omega_T,\rn)$, $u_\delta~ \accentset{\scriptstyle \ast}{\rightharpoondown}~ u$ weakly-$\ast$ in $L^\infty(0,T;L^{m+1}(\Omega))$, $\partial_t u_\delta \rightharpoondown \partial_t u$ weakly in $L^2(0,T;H^{-1}(\Omega))$, and $\nabla u_\delta^\frac{m+1}{2} \rightharpoondown \nabla u^\frac{m+1}{2}$ weakly in $L^2(U_{t_1,t_2},\rn)$, and, by the lower semicontinuity of the $L^2$-norm, we obtain
\begin{align*}
\big\|\nabla u^m \big\|_{L^2(\Omega_T,\rn)} \le \liminf_{\delta\rightarrow 0 +} \big\|\nabla u_\delta^m\big\|_{L^2(\Omega_T,\rn)} \le C_1 
\end{align*}
and
\begin{align*}
\big\|\nabla u^{\frac{m+1}2} \big\|_{L^2(U_{t_1,t_2},\rn)} \le \liminf_{\delta\rightarrow 0 +} \big\|\nabla u_\delta^{\frac{m+1}2} \big\|_{L^2(U_{t_1,t_2},\rn)} \le C_2
\end{align*}
for any $U\Subset\Omega$ and $0<t_1<t_2<T$ with constants $C_1$ and $C_2$ as in \eqref{eq:energy estimates} and \eqref{eq:energy estimates_1}. The fact that $u$ satisfies the variational inequality \eqref{var_ineq_strong} follows from a calculation with the test function $\ph=\alpha \eta (v^m - u_\delta^m + \delta \eta_\delta)$ in the weak formulation \eqref{weak_penalized_PME} of the penalized porous medium equation. Here, $\alpha$ and $\eta$ are cut-off functions in time and space, respectively, as in the definition of local strong solutions to the obstacle problem (see \cite[Def.\ 2.1]{obstacle}), and $\eta_\delta \in C_0^\infty(\Omega)$ is a cut-off function such that 
\[
\begin{cases}
  \eta_\delta = 1 \quad \text{in } \{x\in \Omega \colon \dist(x,\bd \Omega) \ge \delta\},\\
|\nabla \eta_\delta | \le \frac C \delta.
\end{cases}
\]
This shows that $u$ locally solves the obstacle problem and since it attains the correct boundary and initial values, it is also a strong solution to \eqref{obstacle inequality} by \cite[Lemma 3.5]{obstacle}. For the details, we refer to \cite[Section 8]{obstacle}.
\end{proof}

\section{Gradient estimates for weak solutions}\label{sec:grad}
As $\Psi$ is not bounded in general, strong solutions to the obstacle problem might not exist. Hence, we turn our attention to weak solutions. From now on, we will assume \eqref{cond:boundedness}--\eqref{cond:integrability}, but drop the condition \eqref{cond:extra_strong}. Our aim in this section is to prove the gradient estimate \eqref{grad_estimate_weak} for weak solutions by approximating them by strong solutions. Note that, in order to pass to the limit, the energy estimates from Section~\ref{sec:strong} do not suffice because they depend on $\Psi$. 

We start with the following estimate, which can be found in \cite[Lemma~9.1]{obstacle}.

\begin{lemma}\label{weak sol energy estimate}
  Let $u$ be a weak solution to the obstacle problem \eqref{obstacle inequality}. Then, we have
  \begin{align*}
    \sup_{t\in[0,T]} \int_\Omega u(\cdot, t)^{m+1}\, dx + \iint_{\Omega_T} \Big(u^{2m}+ |\nabla u^m|^2 \Big)\, dx \,dt \le C \widetilde A.
  \end{align*}
Here, $C$ is a constant depending on $n,m,\textup{diam}(\Omega)$ and $T$, and $\widetilde A$ is defined as
\begin{align*}
\widetilde A &= \sup_{t\in[0,T]} \int_{\Omega} g(\cdot,t)^{m+1} \, dx + \int_\Omega u_0^{m+1} \, dx \\&~~+ \iint_{\Omega_T} \Big(g^{2m} + |\nabla g^m|^2 + |\partial_t g^m|^{\frac {m+1}m} \Big) \, dx \, dt.
\end{align*}
\end{lemma}

Next, in order to control $\nabla u^{\frac {m+1}2}$, we recall that the strong solutions to the obstacle problem constructed in Lemma \ref{strong solution lemma} are also weak supersolutions to the porous medium equation (see \cite[Thm.\ 2.6]{obstacle}). Thus, the energy estimates for weak supersolutions are at our disposal. Let $U\Subset\Omega$ and $0<t_1<t_2<T$. Then, provided that $u\le M$ for some constant $M>0$, we may apply Lemma~\ref{caccioppoli} with a suitable cut-off function $\zeta\in C_0^\infty(\Omega_T)$ to get
\begin{equation} \label{eq:gradient estimate for strong sols}
\begin{aligned}
  &\iint_{U_{t_1,t_2}} \big|\nabla u^{\frac{m+1}2} \big|^2 \, dx \, dt \\
&~~\le C \iint_{\Omega_T} \Big( (M-u)^2 \zeta |\partial_t \zeta| + u^{m-1} (M-u)^2 |\nabla \zeta|^2 \Big) \, dx \, dt \le C,
\end{aligned}
\end{equation}
where $C$ depends on $M, m, |\Omega_T|$ and $\dist(U_{t_1,t_2}, \bd \Omega_T)$. Here, $|\Omega_T|$ denotes the $(n+1)$-dimensional Lebesgue measure of the set $\Omega_T$.

We proceed to prove the existence of weak solutions to the obstacle problem whose gradients $\nabla u^\frac{m+1}{2}$ are locally bounded in $L^2(\Omega_T,\rn)$. For the existence proof, we reproduce the key ideas from \cite[Section 9]{obstacle} whereas our contribution is establishing the gradient estimate.

\begin{lemma} \label{weak_sol_exist_lemma}
 Suppose that the conditions \eqref{cond:boundedness}--\eqref{cond:integrability} hold. Then, there exists a weak solution to the obstacle problem \eqref{obstacle inequality}, which satisfies $u^{\frac{m+1}2}\in  L^2_{\textup{loc}}(0,T;H^1_{\textup{loc}}(\Omega))$. Moreover, we have the estimate 
\begin{align} \label{grad_estimate_weak}
\big\|\nabla u^{\frac{m+1}2} \big\|_{L^2(U_{t_1,t_2},\rn)} \le C
\end{align}
for any $U\Subset\Omega$ and any $0<t_1<t_2<T$. Here, $C$ is a constant depending on $M, m, |\Omega_T|$ and $\dist(U_{t_1,t_2}, \bd \Omega_T)$, where the constant $M>0$ denotes the upper bounds for $\psi,g$ and $u_0$ from \eqref{cond:boundedness}. Finally, u is also a weak supersolution to the porous medium equation in $\Omega_T$, and if the obstacle $\psi$ is additionally H\"older continuous, then, $u$ is a weak solution to the porous medium equation in the set $\{(x,t)\in\Omega_T\colon u(x,t)>\psi(x,t)\}$.
\end{lemma}

\begin{proof}
Let $U\Subset\Omega$ and $0<t_1<t_2<T$ be fixed. We approximate the obstacle $\psi$ by a sequence of uniformly bounded obstacles $(\psi_i)_{i\in\mathbb N}$ satisfying 
\[
\begin{cases}
\partial_t \psi_i - \Delta \psi_i^m \in L^\infty(\Omega_T),\\
\psi_i^m \rightarrow \psi^m ~~\text{strongly in } L^2(0,T; H^1(\Omega)),\\
\partial_t \psi_i^m \rightarrow \partial_t \psi^m ~~\text{strongly in } L^{\frac {m+1}m} (\Omega_T),\\
\psi_i^m(\cdot,0)\rightharpoondown \psi^m(\cdot,0) ~~\text{weakly in } H^1(\Omega). 
\end{cases}
\]
Note that, unlike in \cite{obstacle}, it is not necessary to approximate $g$ and $u_0$ because of our boundedness assumption \eqref{cond:boundedness}. By Lemma \ref{strong solution lemma}, the obstacle problem with $\psi_i$ as an obstacle has a strong solution $u_i\in C^0([0,T]; L^{m+1}(\Omega))$ satisfying $u_i^m \in g^m + L^2(0,T; H_0^1(\Omega))$ and $u_i^\frac{m+1}{2} \in L^2(t_1,t_2;H^1(U))$. We argue that the functions $u_i$ are uniformly bounded with respect to $i$ by a constant that can be determined in terms of $M$. First, since $\psi_i$ is uniformly bounded, we know that in the contact set $\{ u_i=\psi_i \}$, also $u_i$ is uniformly bounded. Moreover, outside the contact set, $u_i$ is a weak solution to the porous medium equation by \cite[Thm.\ 2.6]{obstacle}. Thus, the comparison principle from \cite[Thm.\ 3.1]{avelin_lukkari_comparison} applied to $u_i$ and the upper bound for $\max\{ \psi_i,g,u_0 \}$ yields that $u_i$ is uniformly bounded in $\Omega_T$. Hence, by Lemma~\ref{weak sol energy estimate} and the Caccioppoli estimate \eqref{eq:gradient estimate for strong sols}, $u_i$ satisfies the estimates
\begin{equation} \label{ui_bounds}
\sup_{t\in[0,T]} \int_\Omega u_i(\cdot,t)^{m+1} \, dx + 
\big\| u_i^m\big\|_{L^2(\Omega_T)} +
\big\| \nabla u_i^m\big\|_{L^2(\Omega_T,\rn)}\le C_1
\end{equation}
and 
\begin{equation} \label{ui_bounds_2}
\big\| \nabla u_i^{\frac{m+1}2} \big\|_{L^2(U_{t_1,t_2},\rn)} \le C_2,
\end{equation}
where $C_1 = C_1(n,m,\diam(\Omega),T,\widetilde A)$ and $C_2 = C_2(m,U,t_1,t_2,\Omega_T,M)$ are constants independent of $i$. Thus, there exist (not relabeled) subsequences $u_i^m$ and $u^{\frac{m+1}2}_i$ which are weakly convergent in $L^2_{\textup{loc}}(0,T;H^1_{\textup{loc}}(\Omega))$. In order to identify the limits, we will establish that $u_i^m\rightarrow u^m$ strongly in $L^2_{\textup{loc}}(\Omega_T)$. For that purpose, we introduce the time regularized functions 
\begin{align*}
w_{i,h}^m= \llbracket u_i^m \rrbracket_h -\llbracket \psi_i^m \rrbracket_h + \psi_i^m ~~\text{ and }~~ w_h^m =  \llbracket u^m \rrbracket_h - \llbracket \psi^m \rrbracket_h + \psi^m,
\end{align*}
where we choose $v_0 = u_0^m$ in \eqref{regularization} to define $\llbracket u_i^m \rrbracket_h$, and analogous choices determine the other mollifications. The first step is showing that $u_i\rightarrow u$ in $L^{m+1}(U^\varrho_T)$, where $U^\varrho_T=B(x_0,\varrho)\times (0,T)$ with a ball $B(x_0,\varrho)\Subset \Omega$. By the triangular inequality, we have
\begin{align*}
&\|u_i-u\|_{L^{m+1}(U^\varrho_T)} \\ &~~\le \|u_i-w_{i,h}\|_{L^{m+1}(U^\varrho_T)} +\|w_{i,h}-w_h\|_{L^{m+1}(U^\varrho_T)}+\|w_{h}-u\|_{L^{m+1}(U^\varrho_T)} \\
&~~=\textup{I} + \textup{II} + \textup{III}.
\end{align*}
We estimate the first term by \cite[ineq.\ (9.17)]{obstacle} to get $\textup{I} \le Ch^\frac{1}{m+1}$ for all $i\in\n$ and $h>0$ with a constant $C$ independent of $i$ and $h$. After that, we use the convergence $w_{i,h}\to w_h \text{ in } L^{2m}(U_T^\varrho)$ from \cite[eq.\ (9.13)]{obstacle} to find that $\textup{II} \to 0$ as $i\to\infty$ for any $h>0$. Finally, the properties of the mollification (see \cite[Lemma 2.2]{kinnunen_lindqvist_reg}) guarantee that $\textup{III} \to 0$ as $h\to 0$.

Therefore, we conclude that $u_i\rightarrow u$ strongly in $L^{m+1}(U^\varrho_T)$, and an elementary computation shows that $u_i^m\rightarrow u^m$ strongly in $L^{\frac{m+1}m}(U^\varrho_T).$ Then, by applying Sobolev's inequality from Lemma \ref{sobolev} to $v=u_i^m$ with $p=2$ and $ r = \frac {m+1}m$, we get
\begin{align*}
&\iint_{U^\varrho_T} u_i^q \, dx \, dt \\&~~ \le C \iint_{U^\varrho_T} \left( \Big|\frac{u_i}{\varrho}\Big|^{2m} + |\nabla u_i^m|^2 \right)\, dx \,dt \left( \sup_{t\in[0,T]} \int_{B(x_0,\varrho)} u_i(\cdot,t)^{m+1}\, dx \right)^{2/n}
\end{align*}
with $q= 2 \left(m+\frac{m+1}n\right) >2m$ and a constant $C$ independent of $i$. By the energy estimates \eqref{ui_bounds} for $u_i$, the right-hand side is uniformly bounded with respect to $i$. Consequently, interpolation tells us $u_i^m \rightarrow u^m$ strongly in $L^2_{\textup{loc}}(\Omega_T)$. Now, we can argue as in the proof of \cite[Thm.\ 2.7]{obstacle} to conclude that $u$ is a weak solution to the obstacle problem in $\Omega_T$ and a weak supersolution to the porous medium equation in $\Omega_T$. Since the bound for $\nabla u_i^\frac{m+1}{2}$ from \eqref{ui_bounds_2} is uniform with respect to $i$, it persists in the limit. 

At this point, it remains to prove that $u$ is a weak solution to the porous medium equation in the set $\{z\in\Omega_T:u(z)>\psi(z)\}$ provided that the obstacle $\psi$ is H\"older continuous. To this aim, we first observe that, following the proof of \cite[Thm.\ 2.7]{obstacle}, it turns out that $u$ is also a local weak solution to the obstacle problem in the sense of \cite[Def.\ 2.1]{obstacle}. Since we assumed that $\psi$ is H\"older continuous, we may apply \cite[Thm.\ 1.1]{Holder_reg_Bog_Luk_Sch} to conclude that $u$ is locally H\"older continuous. Moreover, we know that $u$ is a weak solution to the obstacle problem on any subcylinder $Q\subset\Omega_T$. In the following, we consider a cylinder $Q=U\times[t_1,t_2)\Subset \{z\in\Omega_T\colon u(z)>\psi(z)\}$. Our aim is to prove that $u$ is a weak solution to the porous medium equation in $Q$ with initial datum $u_0=u(\cdot,t_1)$. 
Since $u$ and $\psi$ are continuous on $\overline Q$, there exists $\epsilon>0$ such that $u^m\ge \psi^m+2\epsilon$ in $Q$. We construct the mollification in time $\llbracket u^m \rrbracket_h$ according to \eqref{regularization} subordinate to the cylinder $Q$ with initial values $v_0=u^m(\cdot,t_1)$. By similar arguments as in the proof of \cite[Lemma B.2\,(i)]{BDM-3}, we can show that $\llbracket u^m \rrbracket_h\to u^m$ uniformly in $Q$. Therefore, there exists $\tilde h>0$ such that $\|\llbracket u^m \rrbracket_h - u^m\|_{L^\infty(Q)}\le \epsilon$ for any $h\in(0,\tilde h]$ so that $\llbracket u^m \rrbracket_h\ge\psi^m+\epsilon$ in $Q$. We now consider a function $\varphi\in C_0^\infty(Q)$ and prove that \eqref{weak-super} also holds for this test function. Without loss of generality, we may assume that $\inf_Q\varphi<0$ since otherwise there is nothing to prove. In the variational inequality \eqref{ineq-weak} on $Q$, we choose as comparison function 
$$
	v^m
	=
	\llbracket u^m \rrbracket_h + \tilde \epsilon\varphi,
$$
where 
$$
	0
	<
	\tilde \epsilon
	<
	\frac{\epsilon}{-\inf_Q\varphi}.
$$
With this choice, we have that $v\ge\psi$ on $Q$ and hence, $v$ is admissible in \eqref{ineq-weak}. Therefore, we obtain
\begin{align}\label{coinc}
	&\big\langle\hspace{-.1cm}\big\langle \partial_t u, \alpha \big(\llbracket u^m \rrbracket_h -u^m+ \tilde \epsilon\varphi\big)\big\rangle\hspace{-.1cm}\big\rangle_{u(\cdot,t_1)} \nonumber\\
	&\qquad+ 
	\iint_{Q} \alpha \nabla u^m \cdot \nabla \big(\llbracket u^m \rrbracket_h -u^m+ \tilde \epsilon\varphi\big) \, dx \,dt \ge 0
\end{align}
for all non-negative cut-off functions $\alpha \in W^{1,\infty}([t_1,t_2])$ with $\alpha(t_2)=0$. Here, we choose 
$$
	\alpha(t)
	:=
	\begin{cases}
	1 & \mbox{for $t\in [t_1,t_2-\delta]$}, \\
	\frac{t_2-t}{\delta} & \mbox{for $t\in (t_2-\delta,t_2]$}
	\end{cases}
$$
with $\delta\in(0,t_2-t_1)$. Taking into account 
\begin{align*}
	-\iint_Q & \alpha u \partial_t \llbracket u^m \rrbracket_h \, dx \, dt \\
	&=
	-\iint_Q \alpha \llbracket u^m \rrbracket_h^{\frac1m} \partial_t \llbracket u^m \rrbracket_h \, dx \, dt -
	\iint_Q \alpha \Big[u - \llbracket u^m \rrbracket_h^{\frac1m}\Big] \partial_t \llbracket u^m \rrbracket_h \, dx \, dt \\
	&=
	-\tfrac{m}{m+1} \iint_Q \alpha \partial_t \llbracket u^m \rrbracket_h^{\frac{m+1}{m}} \, dx \, dt -
	\tfrac1h \iint_Q \alpha \Big[u - \llbracket u^m \rrbracket_h^{\frac1m}\Big] \big[u^m - \llbracket u^m \rrbracket_h\big]  dx \, dt \\
	&\le
	\tfrac{m}{m+1} \iint_Q \alpha' \llbracket u^m \rrbracket_h^{\frac{m+1}{m}} \, dx \, dt +
	\tfrac{m}{m+1} \int_U u^{m+1}(\cdot,t_1) \, dx,
\end{align*}
where we used \eqref{magic_formula_mollification}, we may compute for the first integral on the left-hand side of \eqref{coinc} that
\begin{align*}
	&\big\langle\hspace{-.1cm}\big\langle \partial_t u, 
	\alpha \big(\llbracket u^m \rrbracket_h -u^m+ \delta\varphi\big)\big\rangle\hspace{-.1cm}\big\rangle_{u(\cdot,t_1)} \\
	&\qquad\le
	\iint_{Q} \alpha'\left( \tfrac 1 {m+1} u^{m+1} + \tfrac{m}{m+1} \llbracket u^m \rrbracket_h^{\frac{m+1}{m}} -
	u\llbracket u^m \rrbracket_h - \tilde\epsilon u\varphi\right) \, dx \, dt  \\
	&\qquad\quad- \tilde\epsilon\iint_Q \alpha u \partial_t \varphi \, dx \, dt.
\end{align*}
Since $\varphi(\cdot,t_1)=0$ and $\llbracket u^m \rrbracket_h\to u^m$ in $L^{\frac{m+1}{m}}(Q)$, the terms on right-hand side converge to 
\begin{align*}
	- \tilde\epsilon\iint_Q u \partial_t \varphi \, dx \, dt
\end{align*}
as $h,\delta\downarrow 0$. Moreover, the second integral on the left-hand side of \eqref{coinc} tends to 
$$
	\tilde\epsilon\iint_{Q} \nabla u^m \cdot \nabla \varphi \, dx \,dt
$$
in the limit $h,\delta\downarrow 0$. Therefore, we have shown that the inequality \eqref{weak-super} holds for each $\varphi\in C_0^\infty(\Omega_T)$ without any assumption on the sign of $\varphi$. This implies that $u$ is a weak solution to the porous medium equation in $Q$. Since $Q$ was an arbitrary cylinder in the set $\{z\in\Omega_T\colon u(z)>\psi(z)\}$, Lemma~\ref{weak_sol_exist_lemma} is proven. 
\end{proof}

\section{Proof of the main results}\label{sec:proof}

In this section, we will prove Theorem \ref{thm:equivalence} and Theorem \ref{equivalence_celebrated}. We begin with the proof of the former one. The approximation by solutions to the obstacle problem follows the ideas of \cite[Thm.\ 1.3]{supersol}, and the novelty in our paper is establishing gradient estimates which ensure that the regularity property \eqref{integrab_cond} is valid.

\begin{proof}[Proof of Theorem \ref{thm:equivalence}]
Since $u$ is locally bounded and lower semicontinuous by assumption, there exists a sequence of locally uniformly bounded functions $\psi_i\in C^\infty(\Omega_T)$ such that
\[
\psi_i < \psi_{i+1} ~~\text{for any $i\in\n$~~~and}~~~ \lim_{i\rightarrow \infty} \psi_i(x,t) = u(x,t) ~~\text{for a.\,e.\ } (x,t)\in \Omega_T. 
\]
Without loss of generality, we may consider sets $V_{\tau_1,\tau_2} \Subset U_{t_1,t_2}\Subset \Omega_T$. Then, Lemma \ref{weak_sol_exist_lemma} ensures that, for each $i$, there exists a weak solution $u_i$ to the obstacle problem for the porous medium equation in $U_{t_1,t_2}$ with obstacle and initial and lateral boundary data $\psi_i$, i.\,e.\ $u_i$ satisfies
\begin{equation*}
\begin{cases}
u_i^m\in \psi_i^m + L^2(t_1,t_2;H^1_0(U)),\\
u_i\ge \psi_i ~~\text{a.\,e.\ in } U_{t_1,t_2}, 
\end{cases}
\end{equation*}
and
\begin{align*}
\big\langle\hspace{-.1cm}\big\langle \partial_t u_i, \alpha (v^m-u_i^m)\big\rangle\hspace{-.1cm}\big\rangle_{\psi_i(\cdot,t_1)} + \iint_{U_{t_1,t_2}} \alpha \nabla u_i^m \cdot \nabla (v^m-u_i^m) \, dx \,dt \ge 0
\end{align*}
for all comparison maps $v\in \psi_i^m + L^2(t_1,t_2;H^1_0(U))$ with $v\ge\psi_i$ a.\,e.\ in $U_{t_1,t_2}$ and $\partial_t v^m \in L^{\frac {m+1}m}(U_{t_1,t_2})$, and for all non-negative cut-off functions $\alpha \in W^{1,\infty}([t_1,t_2])$ with $\alpha(t_2)=0$. By Lemma \ref{weak_sol_exist_lemma}, for each $i$, $u_i$ is a weak sopersolution to the porous medium equation in $U_{t_1,t_2}$ and a weak solution to the porous medium equation in the set $\{(x,t)\in U_{t_1,t_2} \colon u_i(x,t)>\psi_i(x,t)\}$. Furthermore, by Lemma \ref{weak sol energy estimate} and Lemma \ref{weak_sol_exist_lemma}, the gradients satisfy
\begin{align} \label{gradient_estim_propert}
&\big\| \nabla u_i^m\big\|_{L^2(U_{t_1,t_2},\rn)}\le C ~~\text{ and }~~ \big\| \nabla u_i^{\frac{m+1}2} \big\|_{L^2(V_{\tau_1,\tau_2},\rn)} \le C
\end{align} 
with a constant $C$ that is uniform with respect to $i$. On the other hand, by \cite[Thm.\ 1.1]{Holder_reg_Bog_Luk_Sch} and Remark \ref{rem:sol_supersol}, $u_i$ is a locally continuous weak supersolution in $U_{t_1,t_2}.$ In order to conclude that $u_i\rightarrow u$ in $L^2(V_{\tau_1,\tau_2})$, we will show that 
\begin{align*}
u_i\le u_{i+1}\le u ~~\text{a.\,e.\ in } V_{\tau_1,\tau_2} \text{ for any $i$}.
\end{align*}
To this aim, we consider the sets
\[
K_i=\big\{ (x,t)\in \overline{V_{\tau_1,\tau_2}} : u_i(x,t) \ge \psi_{i+1}(x,t) \big\}.
\]
Since $\psi_{i+1}>\psi_i$ and the functions $u_i$ and $\psi_i$ are continuous in $V_{\tau_1,\tau_2}$ for every $i$, the set $K_i$ is compact. If $K_i = \emptyset$, we have $u_i<\psi_{i+1}$ in $\overline{V_{\tau_1,\tau_2}}$, which implies $u_i<u_{i+1}$ and $u_i<u$. Therefore, it remains to consider the case $K_i\ne \emptyset$. Then, the distance 
\[
d=\dist \left(K_i, \big\{ (x,t)\in \overline{V_{\tau_1,\tau_2}} : u_i(x,t)=\psi_i(x,t) \big\} \right)
\]
is positive. As $K_i$ is compact, there exists a finite number $N$ such that $K_i$ can be covered with $N$ dyadic cubes $Q_j$ with diameter $\diam(Q_j)<d/2$. Hence, we may define
\[
Q=\bigcup_{j=1}^N Q_j \subset \big\{ (x,t)\in \overline{V_{\tau_1,\tau_2}} : u_i(x,t)>\psi_i(x,t) \big\}
\]
and conclude that $u_i$ is a weak solution to the porous medium equation in $Q$ because the contact set $\{u_i=\psi_i\}$ does not intersect $Q$. By construction, we have $u_i\le \psi_{i+1}<u$ on $\bd_p Q$, where $\partial_p$ denotes the natural generalization of the parabolic boundary for finite unions of cylinders (see \cite[Section 3]{avelin_lukkari_comparison} for the exact definition). Since the comparison principle from Def.\ \ref{def:supercaloric}\,(3) can be applied to $u_i$ and $u$ in such a set (see the argument in \cite[Rem.\ 3.4]{lehtela_lukkari}), we deduce that $u_i \le u$ in $Q \supset K_i$. In addition, the inequality $u_i<\psi_{i+1}<u$ holds in $\overline{V_{\tau_1,\tau_2}} \setminus K_i$ by the definition of the set $K_i$. As $\eps>0$ was arbitrary, we infer that $u_i\le u$ in $V_{\tau_1,\tau_2}$. 

On the other hand, we have $u_i \le \psi_{i+1}\le u_{i+1}$ on $\bd_p Q$. Since $u_i$ is a weak solution to the porous medium equation and $u_{i+1}$ is a weak supersolution in $Q$, we may use the comparison principle from \cite[Thm.\ 3.1]{avelin_lukkari_comparison} to find that $u_i\le u_{i+1}$ in $Q \supset K_i$. Again, in the set $\overline{V_{\tau_1,\tau_2}} \setminus K_i$, the inequality $u_i < \psi_{i+1} \le u_{i+1}$ holds by the definition of $K_i$. Thus, we obtain that $u_i\le u_{i+1}$ in $V_{\tau_1,\tau_2}$. 

Collecting the facts, we see that
\begin{align*}
&\psi_i \le u_i \le u_{i+1} \le u ~~\text{a.\,e.\ in $V_{\tau_1,\tau_2}$ for any $i$} ~~\text{ and }~~ \psi_i\rightarrow u.
\end{align*}
From this and the dominated convergence theorem, we know $u_i \to u$ in $L^2(V_{\tau_1,\tau_2})$. It remains to show that $u$ is a weak supersolution satisfying \eqref{integrab_cond}. By \eqref{gradient_estim_propert} and weak compactness, we deduce the existence of the gradients $\nabla u^m$ and $\nabla u^\frac{m+1}{2}$ as well as the convergences
\begin{align*}
&\nabla u_i^m \rightharpoondown \nabla u^m ~~\text{weakly in } L^2(U_{t_1,t_2},\mathbb R^n) ~~ \text{and} \\
&\nabla u_i^\frac{m+1}{2} \rightharpoondown \nabla u^\frac{m+1}{2} ~~\text{weakly in } L^2(V_{\tau_1,\tau_2},\mathbb R^n)
\end{align*}
for (not relabeled) subsequences of $\nabla u_i^m$ and $\nabla u_i^\frac{m+1}{2}$. Since $V_{\tau_1,\tau_2}\Subset \Omega_T$ was arbitrary, we conclude that $\nabla u^m\in L^2_{\rm loc}(\Omega,\mathbb R^n)$ and $\nabla u^\frac{m+1}{2}\in L^2_{\rm loc}(\Omega,\mathbb R^n)$. Finally, due to the convergences $u_i \to u$ in $L^2_{\rm loc}(\Omega_T)$ and $\nabla u_i^m \rightharpoondown \nabla u^m$ weakly in $L^2_{\rm loc}(\Omega_T,\mathbb R^n)$, we may pass to the limit $i\to \infty$ in 
\begin{align*}
		\iint_{\Omega_T} \Big(-u_i\partial_t\ph + \nabla u_i^m \cdot \nabla \ph \Big) \, dx \, dt \ge 0	
\end{align*}
and conclude that $u$ satisfies \eqref{weak-super}. Moreover, using the time mollification \eqref{regularization} we can show by an argument similar to the proof of \cite[Lemma~5.2]{obstacle} that $u\in C^0((0,T);L^{m+1}_{\rm loc}(\Omega))$. This ensures that $u$ is a weak supersolution to the porous medium equation and finishes the proof of Theorem \ref{thm:equivalence}.
\end{proof}

As a consequence of Theorem \ref{thm:equivalence}, we may now prove Theorem \ref{equivalence_celebrated}.

\begin{proof}[Proof of Theorem \ref{equivalence_celebrated}]
First, let $u$ satisfy (i). From \cite[Thm.\ 1.2]{DiBenedetto_Holder} and \cite[Thm.\ 1.1]{dahlberg_kenig}, respectively, we know that $u$ has a locally continuous representative. Without loss of generality, we may assume that $u$ itself is locally continuous and, in particular, locally bounded. Furthermore, due to \cite[Thm.\ 6.5]{vazquez}, the comparison principle holds for $u$. Therefore, we conclude that $u$ is $m$-supercaloric, and hence, Theorem \ref{thm:equivalence} ensures $u^\frac{m+1}{2} \in L^2_\textup{loc}(0,T;H^1_\textup{loc}(\Omega))$. Moreover, the assumption $u\in C^0((0,T);L^{m+1}_{\rm loc}(\Omega))$ implies $u\in C^0((0,T);L^{2}_{\rm loc}(\Omega))$ so that assertion (ii) is verified. 

On the other hand, suppose that $u$ satisfies (ii). Since $u\in L^m_{\rm loc}(\Omega_T)$, we may apply \cite[Thm.\ 1.1]{dahlberg_kenig} to find that $u$ is locally bounded. 
In order to show the existence of $\nabla u^m$, we let $k\in\n$ and consider the truncations $u_k = \min\{ u,k \}$. Since the mapping 
$[0,k]\ni s\mapsto s^\frac{2m}{m+1}$ is Lipschitz continuous, we conclude that $u_k^m$ is weakly differentiable and 
$$
	\nabla u_k^m
	=
	\nabla \Big[\big(u_k^\frac{m+1}{2}\big)^\frac{2m}{m+1}\Big]
	=
	\tfrac{2m}{m+1}u_k^\frac{m-1}{2} \nabla u_k^\frac{m+1}{2}
$$
so that 
\begin{align*}
\iint_{U_{t_1,t_2}} u_k^m \partial_{x_i} \ph\,dx\,dt 
&= 
- \tfrac{2m}{m+1} \iint_{U_{t_1,t_2}} u_k^\frac{m-1}{2} \partial_{x_i} \big( u_k^\frac{m+1}{2} \big) \ph\,dx\,dt 
\end{align*} 
for any $i\in\{1,\dots,n\}$, any $U_{t_1,t_2}\Subset\Omega_T$ and any $\ph \in C_0^\infty(U_{t_1,t_2})$. We recall that, by assumption, we have $u^\frac{m+1}{2}\in L^2_\textup{loc}(0,T;H^1_\textup{loc}(\Omega))$. Therefore, we can infer strong convergence of $u_k$ in $L^{m+1}(U_{t_1,t_2})$ as well as weak subconvergence of $\nabla u_k^\frac{m+1}{2}$ in $L^2(U_{t_1,t_2},\rn)$. This allows us to pass to the limit $k\to\infty$ in the above equation concluding that the weak gradient of $u^m$ exists and is given by $\tfrac{2m}{m+1} u^\frac{m-1}{2} \nabla u^\frac{m+1}{2}$. Consequently, the weak formulations from (i) and (ii) coincide. Moreover, due to the local boundedness of $u$, we obtain
\begin{align*}
\iint_{U_{t_1,t_2}} | \nabla u^m|^2 \,dx\,dt &= \big(\tfrac{2m}{m+1}\big)^2 \iint_{U_{t_1,t_2}} u^{m-1} \big| \nabla u^\frac{m+1}{2}\big|^2 \,dx\,dt \\ &\leq C \iint_{U_{t_1,t_2}} \big| \nabla u^\frac{m+1}{2}\big|^2 \,dx\,dt <\infty,
\end{align*}
where the constant $C$ depends on $m$ and the upper bound for $u$ in $U_{t_1,t_2}$. 
This ensures that $\nabla u^m\in L^2(U_{t_1,t_2},\rn)$. Together with the fact that $u$ satisfies the weak formulation of the porous medium equation, this implies that $u\in C^0([t_1,t_2);L^{m+1}_{\rm loc}(U))$; cf. \cite[Lemma 5.2]{obstacle}. 
Since $U_{t_1,t_2}\Subset\Omega_T$ was arbitrary, we deduce that (i) holds, which finishes the proof.
\end{proof} 

\begin{proof}[Alternative proof of Theorem \ref{equivalence_celebrated}, $(i) \Rightarrow (ii)$]
Let $u$ satisfy (i). 
For $\epsilon>0$, we let $g_\epsilon(s):=\max\{\epsilon,s\}=\max\{\epsilon^m,s^m\}^{\frac{1}{m}}$ for any $s\geq 0$. Since the mapping $\mathbb R\ni \sigma\mapsto \max\{\epsilon^m,\sigma\}^{\frac1m}$ is Lipschitz continuous, we know that $g_\epsilon(u)=\max\{\epsilon^m,u^m\}^{\frac1m}$ is weakly differentiable and 
\begin{align*}
	\nabla g_\epsilon(u)
	=
	\tfrac1m \rchi_{\{u>\epsilon\}} u^{1-m} \nabla u^m
	\in L^2_{\rm loc}(\Omega_T,\mathbb R^n)
\end{align*}
so that $g_\epsilon(u)\in L^2_\textup{loc}(0,T;H^1_\textup{loc}(\Omega))$. Moreover, we recall the mollification in time defined in \eqref{regularization}, where, throughout this proof, we choose $v_0=0$ as initial value. Similar to inequality \eqref{ineq:regularized} for weak supersolutions, we can derive the following regularized version of \eqref{weak-eq}:
\begin{equation}\label{eq:regularized2}
\iint_{\Omega_T} \Big( \bd_t \llbracket u \rrbracket_h  \ph + \nabla \llbracket u^m \rrbracket_h \cdot \nabla \ph \Big)\, dx \,dt =  \frac 1 h \int_\Omega u(\cdot,0) \int_0^T \ph e^{-\frac s h } \, ds \, dx
\end{equation}
for any test function $\ph\in L^2(0,T;H_0^1(\Omega))$. 
In \eqref{eq:regularized2}, we insert $\varphi = \zeta^2 g_\epsilon(u)$, where $\zeta\in C_0^\infty(\Omega_T)$ is a smooth cut-off function with $0\le \zeta \le 1$. We first treat the evolutionary integral. Using formula \eqref{magic_formula_mollification} for the time derivative of the mollification, integrating by parts, and then, passing to the limit $h\to 0$, we find
\begin{align*}
\iint_{\Omega_T} & \zeta^2  \bd_t \llbracket u \rrbracket_h \;\! g_\epsilon(u)\,dx\,dt \\
&= 
\iint_{\Omega_T}  \zeta^2  \bd_t \llbracket u \rrbracket_h \;\! g_\epsilon\big(\llbracket u \rrbracket_h\big)\,dx\,dt +
\iint_{\Omega_T}  \zeta^2  \bd_t \llbracket u \rrbracket_h \;\! 
\Big(g_\epsilon(u) - g_\epsilon\big(\llbracket u \rrbracket_h\big) \Big)\,dx\,dt \\
&\geq
\iint_{\Omega_T}  \zeta^2  \;\!
\bd_t \bigg(\int_0^{\llbracket u \rrbracket_h} g_\epsilon(s)\,ds\bigg)\,dx\,dt 
= 
-2\iint_{\Omega_T} \zeta\partial_t\zeta\;\! \int_0^{\llbracket u \rrbracket_h} g_\epsilon(s)\,ds \,dx\,dt \\
&= -2 \iint_{\Omega_T} \zeta\partial_t\zeta\;\! G_\epsilon\big(\llbracket u \rrbracket_h\big) \,dx\,dt
\to -2 \iint_{\Omega_T} \zeta\partial_t\zeta\;\! G_\epsilon(u) \,dx\,dt,
\end{align*}
where 
\begin{equation*}
	G_\epsilon(s)
	:=
	\begin{cases}
		\epsilon s& \mbox{for $0\le s\le\epsilon,$}\\
		\frac12(\epsilon^2+s^2)&\mbox{for $s>\epsilon$.}
	\end{cases}
\end{equation*}
Note that the right-hand side term in \eqref{eq:regularized2} vanishes as $h\to 0$, and, for the diffusion term, we get in the limit $h\to 0$ that 
\begin{align*}
	&\iint_{\Omega_T} \nabla \llbracket u^m \rrbracket_h \cdot 
	\nabla \big(\zeta^2 g_\epsilon(u)\big) \,dx\,dt
	\to
	\iint_{\Omega_T} \nabla u^m \cdot 
	\nabla \big(\zeta^2 g_\epsilon(u)\big) \,dx\,dt \\
	&\qquad=
	2\iint_{\Omega_T}  \zeta\;\! g_\epsilon(u) \nabla u^m \cdot \nabla\zeta\,dx\,dt +
	\iint_{\Omega_T} \zeta^2\;\! \nabla u^m \cdot\nabla g_\epsilon(u) \,dx\,dt .
\end{align*}
For the first integral on the right-hand side, we compute
\begin{align*}
	\bigg|\iint_{\Omega_T}  
	\zeta\;\! g_\epsilon(u) \nabla u^m \cdot \nabla\zeta\,dx\,dt \bigg|
	&\le
	\iint_{\Omega_T} g_\epsilon(u) |\nabla u^m| |\nabla\zeta|\,dx\,dt \\
	&\le
	\iint_{\Omega_T} \big[|\nabla\zeta|^2 g_\epsilon(u)^2 + |\nabla u^m|^2\big] \,dx\,dt,
\end{align*}
while, for the second one, we find
\begin{align*}
	\iint_{\Omega_T} \zeta^2\;\! \nabla u^m \cdot\nabla g_\epsilon(u) \,dx\,dt 
	&=
	\tfrac1m\iint_{\Omega_T\cap \{u>\epsilon\}} \zeta^2\;\! u^{1-m}\nabla u^m \cdot\nabla u^m \,dx\,dt \\
	&=
	\tfrac{4m}{(m+1)^2} \iint_{\Omega_T\cap\{u>\epsilon\}}\! \zeta^2 \big| \nabla u^\frac{m+1}{2} \big|^2 dx\,dt,
\end{align*}
where we have defined $\nabla u^\frac{m+1}{2}:=\frac{m+1}{2m}\rchi_{\{u>0\}}u^\frac{1-m}{2}\nabla u^m$. 
Combining the preceding computations, we see that 
\begin{align*}
	&\iint_{\Omega_T\cap\{u>\epsilon\}}\! \zeta^2 \big| \nabla u^\frac{m+1}{2} \big|^2 dx\,dt \\
	&\qquad\le 
	\tfrac{(m+1)^2}{2m} \iint_{\Omega_T} 
	\Big[|\partial_t\zeta| G_\epsilon(u) + |\nabla\zeta|^2g_\epsilon(u)^2 + |\nabla u^m|^2\Big] \,dx\,dt.
\end{align*}
Note that the right-hand side is bounded uniformly with respect to $\epsilon$ and converges as $\epsilon\downarrow 0$. Therefore, we obtain by Fatou's lemma that 
\begin{align*}
	&\iint_{\Omega_T}\! \zeta^2 \big| \nabla u^\frac{m+1}{2} \big|^2 dx\,dt \\
	&\qquad\le
	\liminf_{\epsilon\downarrow 0}
	\iint_{\Omega_T\cap\{u>\epsilon\}}\! \zeta^2 \big| \nabla u^\frac{m+1}{2} \big|^2 dx\,dt \\
	&\qquad\le 
	\lim_{\epsilon\downarrow 0} \tfrac{(m+1)^2}{2m} \iint_{\Omega_T} 
	\Big[|\partial_t\zeta| G_\epsilon(u) + |\nabla\zeta|^2g_\epsilon(u)^2 + |\nabla u^m|^2\Big] \,dx\,dt \\
	&\qquad=
	\tfrac{(m+1)^2}{2m} \iint_{\Omega_T} 
	\Big[\big( \tfrac 12 |\partial_t\zeta| + |\nabla\zeta|^2\big)u^2 + |\nabla u^m|^2\big] \,dx\,dt.
\end{align*}
This ensures that $\nabla u^\frac{m+1}{2}\in L^2_{\rm loc}(\Omega_T,\mathbb R^n)$. It remains to show that $\nabla u^\frac{m+1}{2}$ is the weak derivative of $u^\frac{m+1}{2}$. To this aim, we consider $\ph \in C_0^\infty(\Omega_T)$ and compute for $i\in\{1,\dots,n\}$ that 
\begin{align*}
\iint_{\Omega_T}  u^\frac{m+1}{2} \partial_{x_i} \ph\,dx\,dt 
&= 
\iint_{\Omega_T}  (u^m)^\frac{m+1}{2m} \partial_{x_i} \ph\,dx\,dt \\
&= 
- \tfrac{m+1}{2m} \iint_{\Omega_T} \rchi_{\{u>0\}}u^\frac{1-m}{2} \partial_{x_i} (u^m) \ph\,dx\,dt \\
&= 
- \iint_{\Omega_T} \nabla u^\frac{m+1}{2}\cdot e_i\, \ph\,dx\,dt ,
\end{align*} 
which proves the assertion that $\nabla u^\frac{m+1}{2}$ is the weak derivative of $u^\frac{m+1}{2}$, and hence $u^\frac{m+1}{2} \in L^2_\textup{loc}(0,T;H^1_\textup{loc}(\Omega))$. Moreover, the assumption $u\in C^0((0,T);L^{m+1}_{\rm loc}(\Omega))$ implies $u\in C^0((0,T);L^{2}_{\rm loc}(\Omega))$. Finally, since $\nabla u^m=\frac{m+1}{2m}u^\frac{m-1}2\nabla u^\frac{m+1}{2}$ by the very definition of $\nabla u^\frac{m+1}{2}$, we can deduce \eqref{weak-eq-it} from \eqref{weak-eq} so that assertion~(ii) is verified. 
\end{proof}


\bibliography{citations}
\bibliographystyle{plain}
\end{document}